\documentclass[a4paper,11pt]{article}
\usepackage{amsmath}
\usepackage{amsthm}
\usepackage{amssymb}
\usepackage{mathtools}
\usepackage{mathrsfs}
\usepackage{physics}
\usepackage{enumitem}
\usepackage{stmaryrd} 
\usepackage{xspace}
\usepackage[margin=1.0in]{geometry}
\usepackage[ruled,vlined]{algorithm2e}
\usepackage{cancel}
\usepackage[hidelinks]{hyperref}

\newcommand\N{\mathbb N}

\newcommand\R{\mathbb R}

\newcommand\PP{\mathbb P}
\newcommand\E{\mathbb E}
\newcommand\T{\mathbb T}

\newcommand\Ind{\boldsymbol 1}

\DeclarePairedDelimiter{\floor}{\lfloor}{\rfloor}

\DeclarePairedDelimiterX{\infdivx}[2]{(}{)}{%
  #1\;\delimsize|\delimsize|\;#2%
}

\newcommand{\PR}[1]{\PP\left(#1\right)}

\newcommand{\EX}[1]{\E\left[#1\right]}

\usepackage{xcolor}

\theoremstyle{plain}
\newtheorem{theorem}{Theorem}[section]
\newtheorem*{theorem*}{Theorem}
\newtheorem{lemma}[theorem]{Lemma}
\newtheorem*{lemma*}{Lemma}

\newtheorem*{corollary*}{Corollary}
\newtheorem{proposition}[theorem]{Proposition}
\newtheorem*{proposition*}{Proposition}

\newtheorem*{fact*}{Fact}

\theoremstyle{definition}

\newtheorem{definition*}[theorem]{Definition}
\newtheorem{assumption}[theorem]{Assumption}

\theoremstyle{remark}
\newtheorem{remark}[theorem]{Remark}
\newtheorem*{remark*}{Remark}

\numberwithin{equation}{section}

\allowdisplaybreaks

\begin{document}

\title{Negative moments of the CREM partition function in the high temperature regime}
\author{
Fu-Hsuan~Ho\thanks{
  Institut de Math\'{e}matiques de Toulouse, CNRS UMR5219.
  \textit{Postal address:} Institut de Math\'{e}matiques de Toulouse,
  Universit\'{e} de Toulouse,
  118 Route de Narbonne, 31062 Toulouse Cedex 9, France.
  \textit{Email:} \texttt{fu-hsuan.ho AT math.univ-toulouse.fr}
  }
}
\date{\today}
\maketitle
\begin{abstract}
The continuous random energy model (CREM) was introduced by Bovier and Kurkova in 2004 which can be viewed as a generalization of Derrida's generalized random energy model. Among other things, their work indicates that there exists a critical point $\beta_c$ such that the partition function exhibits a phase transition. The present work focuses on the high temperature regime where $\beta<\beta_c$. We show that for all $\beta<\beta_c$ and for all $s>0$, the negative $s$ moment of the CREM partition function is comparable with the expectation of the CREM partition function to the power of $-s$, up to constants that are independent of $N$.

 
\paragraph{MSC2020 subject classifications:} 60K35; 60G15; 60J80. 
\end{abstract}

\section{Introduction} \label{sec:intro}

In this paper, we consider the continuous random energy model (CREM) introduced by Bovier and Kurkova in \cite{CREM04} based on previous work by Derrida and Spohn \cite{DerridaSpohn}. The model is defined as follows: Let $N\in\N$. Denote by $\T_N$ the binary tree with depth $N$. Given $u\in\T_N$, we denote by $\abs{u}$ the depth of $u$, and we write $w\leq u$ if $w$ is an ancestor of $u$. For all $u,v\in\T_N$, let $u\wedge v$ be the most recent common ancestor of $u$ and $v$. The CREM is a centered Gaussian process $(X_u)_{u\in\T_N}$ indexed by the binary tree $\T_N$ of depth $N$ with covariance function
\begin{align*}
\EX{X_vX_w} = N\cdot A\left(\frac{\abs{v\wedge w}}{N}\right), \quad \forall v,w\in \T_N.
\end{align*}

To study the CREM, one of the key quantities is the partition function defined as
\begin{align}
\label{chap3.eq:par func}
Z_{\beta,N} \coloneqq \sum_{\abs{u}=N} e^{\beta X_u}.
\end{align}
In this paper, we study the negative moments of the partition function which gives us information on the small values of the partition function. This type of study was conducted in the context of other related models such as the homogeneous branching random walks, the multiplicative cascades, or the Gaussian multiplicative chaos, and came with the name of \emph{negative moments}, \emph{left tail} behavior, \emph{small deviations} of the partition function. We give a survey of these results in Section~\ref{sec:proof strategy}. 

\subsection{Main result} \label{sec:main result}

For this paper, we require that the function $A$ satisfies the following assumption.
\begin{assumption} 
\label{chap3.assumption}
We suppose that the function $A$ is a non-decreasing function defined on the interval $[0,1]$ such that $A(0)=0$ and $A(1)=1$. Let $\hat{A}$ be the concave hull of $A$, and we denote by $\hat{A}'$ the right derivative of $\hat{A}$. Throughout this paper, we assume the following regularity conditions.
\begin{enumerate}[label = (\roman*)]
    \item The function $A$ is differentiable in a neighborhood of $0$, i.e., there is an $x_0\in (0,1]$ such that $A$ is differentiable on the interval $(0,x_0)$. Furthermore, we assume that $A$ has a finite right derivative at $0$.
    \item There exists $\alpha\in (0,1)$ such that the derivative of $A$ is locally H\"{o}lder continuous with exponent $\alpha$ in a neighborhood of $0$, i.e., there exists $x_1\in (0,1]$ such that 
    \[\sup_{\substack{x,y\in [0,x_1] \\ x\neq y}} \frac{\abs{A'(x)-A'(y)}}{\abs{x-y}^\alpha}<\infty.\]
    \item The right derivative of $\hat{A}$ at $x=0$ is finite, i.e., $\hat{A}'(0) < \infty$.
\end{enumerate}
\end{assumption}
\begin{remark}
    Point (i)---(iii) are technical conditions that are required in the moment estimates perform in Section~\ref{chap3.sec:useful} and Section~\ref{chap3.sec:moment estimates} so that the estimates only depend on $\hat{A}'(0)$.
\end{remark}
The free energy of the CREM is defined as 
\begin{align*}
F_\beta \coloneqq \lim_{N\rightarrow\infty} \frac{1}{N}\EX{\log Z_{\beta,N}}.
\end{align*}
The free energy $F_\beta$ admits an explicit expression. Namely, for all $\beta\geq 0$,
\begin{align}
F_\beta = \int_0^1 f\left(\beta\sqrt{\hat{A}'(t)}\right) \dd{t}, 
\quad \text{where} \quad
f(x) \coloneqq 
\begin{cases}
\displaystyle \frac{x^2}{2} + \log 2, & x< \sqrt{2\log 2}, \\
\\
\sqrt{2\log 2}x, & x\geq \sqrt{2\log 2}.
\end{cases}
\label{chap3.eq:F_beta}
\end{align}
Formula \eqref{chap3.eq:F_beta} was proven by Bovier and Kurkova in Theorem 3.3 of \cite{CREM04}, based on a Gaussian comparison lemma (Lemma 3.2 of \cite{CREM04}) and previous work of Capocaccia, Cassandro and Picco \cite{GREMFreeEnergy} in the 1980s. While Bovier and Kurkova assumed the function $A$ to be piecewise smooth (Line 4 after (1.2) in \cite{CREM04}) to simplify their article, the proof of Theorem 3.3 of \cite{CREM04} does not use this regularity assumption and the proof also holds for the class of functions that are non-decreasing on the interval $[0,1]$ such that $A(0)=0$ and $A(1)=1$.

In the same paper,  Bovier and Kurkova also showed that the maximum of the CREM satisfies the following.
\begin{align}
\lim_{N\rightarrow\infty }\frac{1}{N} \EX{\max_{\abs{u}=N} X_u} = \sqrt{2\log 2} \int_0^1 \sqrt{\hat{A}'(t)} \dd{t}.
\label{chap3.eq:max}
\end{align}
Combining \eqref{chap3.eq:F_beta} and \eqref{chap3.eq:max} indicates that there exists 
\begin{align}
    \beta_c \coloneqq \frac{\sqrt{2\log2}}{\sqrt{\hat{A}'(0)}} \label{chap3.eq:threshold}
\end{align}
such that the following phase transition occurs. a) For all $\beta< \beta_c$, the main contribution to the partition function comes from an exponential amount of the particles. b) For all $\beta>\beta_c$, the maximum starts to contribute significantly to the partition function. The quantity $\beta_c$ is sometimes referred to as the static critical inverse temperature of the CREM. In the following, we refer to the subcritical regime $\beta<\beta_c$ as the high temperature regime.

Our goal is to study the negative moments of the partition function $Z_{\beta,N}$ in the high temperature regime, and we obtain the following result.
\begin{theorem}
\label{chap3.thm:main}
Suppose Assumption~\ref{chap3.assumption} is true. Let $\beta<\beta_c$. For all $s>0$, there exist $N_0 = N_0(A,\beta,s)\in\N$ and a constant $C=C(A,\beta,s)$, independent of $N$, such that for all $N\geq N_0$, 
\[\EX{(Z_{\beta,N})^{-s}} \leq C\EX{Z_{\beta,N}}^{-s}.\]
\end{theorem}
\begin{remark}
For all $\beta>0$, $N\in\N$ and $s>0$, we have the trivial lower bound provided by the Jensen inequality and the convexity of $x\mapsto x^{-s}$,
\begin{align}
\EX{(Z_{\beta,N})^{-s}} \geq \EX{Z_{\beta,N}}^{-s}. \label{chap3.eq:main.lowerbound}
\end{align} 
Thus, combining \eqref{chap3.eq:main.lowerbound} with Theorem~\ref{chap3.thm:main}, we see that for all $\beta<\beta_c$ in the high temperature regime and for all $s>0$, $\EX{(Z_{\beta,N})^{-s}}$ is comparable with $\EX{Z_{\beta,N}}^{-s}$.
\end{remark}

\subsection{Related models} \label{sec:proof strategy}

For the (homogeneous) branching random walk, the typical approach is not to study the partition function directly but to consider the additive martingale 
\begin{align*}
W_{\beta,N}=Z_{\beta,N}/\EX{Z_{\beta,N}},   
\end{align*}
which converges to a pointwise limit $W_{\beta,\infty}$. A standard method to establish that $W_{\beta,\infty}$ has negative moments involves the following observation. Suppose that $Y'$, $Y''$, $W'$ and $W''$ are independent random variables such that $Y'$ and $Y''$ share the law of the increment of the branching random walk, and each of $W'$ and $W''$ follows the same distribution as $W_{\beta,\infty}$. Then, the limit of the additive martingale $W_{\beta,\infty}$ satisfies the following fixed point equation in distribution:
\begin{align}
W_{\beta,\infty} \stackrel{(d)}{=} Y'W' + Y''W''. \label{chap3.eq:smoothing}
\end{align}
Formula \eqref{chap3.eq:smoothing} is a special case of the so-called smoothing transform, which has been studied extensively in the literature. In the context of multiplicative cascades, which is an equivalent model of the branching random walk, Molchan showed in \cite{molchanScalingExponentsMultifractal1996} that if $\EX{(Y')^{-s}}<\infty$ for some $s>0$, then $\EX{(W_{\beta,\infty})^{-s}}<\infty$. Subsequently, Molchan's result was extended by Liu in \cite{Liu01, liuExtensionFunctionalEquation2002}. More recently, Hu studied in \cite{huHowBigMinimum2016} the small deviation of the maximum of the branching random walk based on Liu's result. On the other hand, Nikula provided small deviations for lognormal multiplicative cascades in \cite{nikulaSmallDeviationsLognormal2020}, thereby refining Molchan's result. 

In the context of the Gaussian multiplicative chaos, Garban, Holden, Sep\'{u}lveda and Sun showed in \cite{garbanNegativeMomentsGaussian2018} that given a subcritical GMC measure with mild conditions on the base measure, the total mass of this GMC measure has negative moments of all orders. Their works expanded on the findings of Robert and Vargas \cite{robertGaussianMultiplicativeChaos2010}, Duplantier and Sheffield \cite{duplantierLiouvilleQuantumGravity2011} and Remy \cite{remyFyodorovBouchaudFormula2020}, with the base measure taken as the Lebesgue measure restricted to some open set, where these authors were interested in quantifying the left tail behavior of the total mass of the GMC near $0$. A final remark is that Robert and Vargas referred the proof of Theorem 3.6 in their paper \cite{robertGaussianMultiplicativeChaos2010}, which concerned the negative moment of the total mass of the GMC, to Barral and Mandelbrot \cite{barralMultifractalProductsCylindrical2002}, where they studied a log-Poisson cascade.


\subsection{Proof strategy}

For the CREM in general, $W_{\beta,N}$ is not a martingale, so there is no obvious way to show that  $W_{\beta,\infty}$ exists. To prove Theorem~\ref{chap3.thm:main}, we adapt the proof of Lemma A.3 in the paper \cite{BenSchramm09} by Benjamini and Schramm where they study the multiplicative cascade.

While their argument also involved the smoothing transform, it can be adjusted for general CREM. In particular, we show that for all $s>0$, there exist two positive sequences $\varepsilon_k$ and $\eta_k$ that both decay double exponentially to $0$ as $k\rightarrow\infty$ such that for $N$ sufficiently large, 
\begin{align}
\PR{Z_{\beta,N}\leq \varepsilon_k\EX{Z_{\beta,N}}} \leq \eta_k, \label{chap3.eq:left tail}
\end{align} 
and that for all $k\in\llbracket 1, C'\log N\rrbracket$, where $C'>0$, there exist $C=C(A,\beta,s)>0$ and $c=c(A,\beta,s)>0$ such that
\begin{align}
\varepsilon_k^{-s}\eta_k\leq C e^{-c e^{c k}}. \label{chap3.eq:double exponential decay}
\end{align}
The proof of \eqref{chap3.eq:left tail} and \eqref{chap3.eq:double exponential decay} is by establishing an initial left tail estimate, and then using the branching property to bootstrap the estimate. 


\paragraph{Outline.} The rest of the paper is organized as follows. We provide in Section~\ref{chap3.sec:init ineq} an initial estimate of the left tail of the partition function. Next, in Section~\ref{chap3.sec:bootstrap}, we improve the estimate obtained in the previous section using a bootstrap argument. Finally, Theorem~\ref{chap3.thm:main} is proven in Section~\ref{chap3.sec:proof of main}.

\section{Initial estimate for the left tail of the partition function} \label{chap3.sec:init ineq}

The main goal of this section is to estimate the left tail of $Z_{\beta,N}$. The main idea is to adapt \eqref{chap3.eq:smoothing} to the inhomogeneous setting via a recursive argument of $Z_{\beta,N}$ which we will explain immediately with the following notation.

We denote by $\llbracket m_1,m_2\rrbracket = \{m_1,m_1+1,\ldots,m_2-1,m_2\}$ for two integers $m-1<m_2$. Fix $k\in\llbracket 0,N\rrbracket$. Define $(X^{(k)}_u)_{\abs{u}=N-k}$ to be the centered Gaussian process with covariance function
\begin{align*}
    \EX{X^{(k)}_uX^{(k)}_w} = N\left(A\left(\frac{\abs{v\wedge w}+k}{N}\right)-A\left(\frac{k}{N}\right)\right).
\end{align*}
We also introduce the corresponding partition function.
\begin{align*}
    Z^{(k)}_{\beta,N-k} = \sum_{\abs{u}=N-k} e^{\beta X^{(k)}_u}.
\end{align*}
By the covariance function of $(X^{(k)}_u)_{\abs{u}=N-k}$ and the definition of $Z^{(k)}_{\beta,N-k}$, we obtain
\begin{align}
    \EX{Z^{(k)}_{\beta,N-k}} 
    = 2^{N-k}\exp(\frac{\beta^2 N}{2}\left(1-A\left(\frac{k}{N}\right)\right)). \label{eq:expectation of part}
\end{align}

With the notation above, $Z_{\beta,N}$ admits the following one-step decomposition
\begin{align*}
    Z_{\beta,N} = \sum_{\abs{u}=1} e^{\beta X_u} Z^{u}_{\beta,N-1}
\end{align*}
where $(Z^{u}_{\beta,N})_{\abs{u}}$ are independent and has the same distribution as $Z^{(1)}_{\beta,N-1}$. Then, the task of controlling the left tail of $Z_{\beta,N}$ can be reduced to controlling $Z^{(1)}_{\beta,N-1}$ and $e^{X_u}$, for $\abs{u}=1$. Then, we decompose $Z^{(1)}_{\beta,N-1}$ in the same manner. We call this procedure the bootstrap argument, which will be made precise in Section~\ref{chap3.sec:bootstrap}. The point is that for the bootstrap argument to work, it requires an initial bound, independent of $N$ large and $k$ small, for the left tail of the partition function $Z^{(k)}_{\beta,N-k}$ in the high-temperature regime. This bound is provided in the following.
\begin{proposition}
\label{chap3.prop:init}
Let $\beta<\beta_c$. Let $K=\floor{2\log_\gamma N}$, where $\gamma\in (11/10,2)$. Then, there exist $N_0=N_0(A,\beta)\in\N$ and a constant $\eta_0 = \eta_0(A,\beta) < 1$ such that for all $N\geq N_0$ and $k\in \llbracket 0,K\rrbracket$,
\begin{align*}
\PR{Z^{(k)}_{\beta,N-k} \leq \frac{1}{2}\cdot \EX{Z^{(k)}_{\beta,N-k}}} \leq \eta_0,
\end{align*}
where the expectation of $Z^{(k)}_{\beta,N-k}$ is given in \eqref{eq:expectation of part}.
\end{proposition}
To prove Proposition~\ref{chap3.prop:init}, one may want to apply the Paley--Zygmund type argument directly to $Z^{(k)}_{\beta,N-k}$. This, however, will not provide the desired result for the whole high temperature regime due to the correlation of the model. Happily, this can be resolved by introducing a good truncated partition function which is defined as follows. 
For all $\abs{u}=N$, $k\in\llbracket 0,N\rrbracket$, $a>0$ and $b>0$, define the truncating set
\[G_{u,k} \coloneqq 
\Bigg\{\forall n\in \llbracket 1,N-k \rrbracket
\,\Bigg|
\begin{array}{l}
     X_w^{(k)}\leq \beta\cdot N\cdot \left(A\left(\frac{n+k}{N}\right)-A\left(\frac{k}{N}\right)\right) + a\cdot n + b   \\
     \ \text{where $w$ is an ancestor of $u$ with $\abs{w}=n$}
\end{array}
\Bigg\}.\]
Define the truncated partition function
\begin{align}
Z^{(k),\leq}_{\beta,N-k} \coloneqq \sum_{\abs{u}=N-k} e^{\beta X^{(k)}_u} \Ind_{G_{u,k}}. \label{chap3.eq:def of truncated partition function}
\end{align}
In Section~\ref{chap3.sec:moment estimates}, we will prove that the truncated partition has matching squared first moment and second moment. Once we have these two estimates, by passing to a Paley--Zygmund type argument, we can estimate the left tail of the untruncated partition function. 

\subsection{Many-to-one and Many-to-two} \label{sec:many-to-one/two}

To perform the first and second moment estimates, we state and prove the following many-to-one/two lemmas.
\begin{lemma}[Many-to-one]
    \label{lem:many-to-one}
    Define $(X^{(k)}_n)_{n\in \llbracket 0,N-k\rrbracket}$ to be a non-homogeneous random walk such that for all $n\in \llbracket 0,N-k-1\rrbracket$, the increment $X^{(k)}_{n+1}-X^{(k)}_n$ is distributed as a centered Gaussian random variable with variance $N(A((n+1)/N)-A(n/N))$.
    Let 
    \begin{align*}
        G_{k} = \left\{\forall n\in \llbracket 1,N-k \rrbracket : X_n^{(k)}\leq \beta\cdot N\cdot \left(A\left(\frac{n+k}{N}\right)-A\left(\frac{k}{N}\right)\right) + a\cdot n + b \right\}
    \end{align*}
    Then, we have 
    \begin{align*}
        \EX{Z^{(k),\leq}_{\beta,N-k}}
        = 2^{N-k} \EX{e^{\beta X_{N-k}^{(k)}} \Ind_{G_{k}}}.
    \end{align*}
\end{lemma}
\begin{proof}
By the linearity of expectation, we have 
\begin{align*}
    \EX{Z^{(k),\leq}_{\beta,N-k}} = \sum_{\abs{u}=N-k} \EX{e^{\beta X_{u}^{(k)}} \Ind_{G_{u,k}}}.
\end{align*}
On the other hand, for each $u$ with $\abs{u}=N-k$, the path $(X_{w}^{(k)})_{w\leq u}$ has the same distribution as $(X^{(k)}_n)_{n\in \llbracket 0,N-k\rrbracket}$, and the proof follows.
\end{proof}

\begin{lemma}[Many-to-two]
    \label{lem:many-to-two}
    Define $(X^{(k)}_{\ell,n})_{n\in \llbracket 0,N-k\rrbracket}$ and $(\tilde{X}^{(k)}_{\ell,n})_{n\in \llbracket 0,N-k\rrbracket}$ to be non-homogeneous random walks such that 
    \begin{itemize}
        \item for all $n\in \llbracket 0,N-k-1\rrbracket$, the increments $X^{(k)}_{\ell,n+1}-X^{(k)}_{\ell,n}$ and $\tilde{X}^{(k)}_{\ell,n+1}-\tilde{X}^{(k)}_{\ell,n}$ are both distributed as a centered Gaussian random variable with variance $N(A((n+1)/N)-A(n/N))$,
        \item $X^{(k)}_{\ell,n}=\tilde{X}^{(k)}_{\ell,n}$ for all $n\in\llbracket 0,\ell\rrbracket$ and 
        \item $(X^{(k)}_{\ell,\ell+m}-X^{(k)}_{\ell,\ell})_{m\in\llbracket 0,N-k-\ell\rrbracket}$ and $(\tilde{X}^{(k)}_{\ell,\ell+m}-\tilde{X}^{(k)}_{\ell,\ell})_{m\in\llbracket 0,N-k-\ell\rrbracket}$ are independent.
    \end{itemize}
    Let 
    \begin{align*}
        G_{\ell,k} &= \left\{\forall n\in \llbracket 1,N-k \rrbracket : X_{\ell,n}^{(k)}\leq \beta\cdot N\cdot \left(A\left(\frac{n+k}{N}\right)-A\left(\frac{k}{N}\right)\right) + a\cdot n + b \right\}, \\
        \tilde{G}_{\ell,k} &= \left\{\forall n\in \llbracket 1,N-k \rrbracket : \tilde{X}_{\ell,n}^{(k)}\leq \beta\cdot N\cdot \left(A\left(\frac{n+k}{N}\right)-A\left(\frac{k}{N}\right)\right) + a\cdot n + b \right\}.
    \end{align*}
    Then, we have
    \begin{align*}
        &\EX{(Z^{(k),\leq}_{\beta,N-k})^2} \nonumber \\
        &= 2^{N-k} \EX{e^{\beta X_{N-k}^{(k)}} \Ind_{G_{k}}} 
        + \sum_{\ell=0}^{N-k-1} 2^{2(N-k)-\ell-1} \EX{e^{\beta (X^{(k)}_{\ell,N-k}+\tilde{X}_{\ell,N-k}^{(k)})} \Ind_{G_{\ell,k}\cap \tilde{G}_{\ell,k}}}.
    \end{align*}
\end{lemma}
\begin{proof}
    By the linearity of expectation, we have the decomposition
    \begin{align}
        &\EX{(Z^{(k),\leq}_{\beta,N-k})^2} \nonumber \\
        &=
        \sum_{\abs{u}=\abs{w}=N-k} \EX{e^{\beta (X^{(k)}_{u}+X_{w}^{(k)})} \Ind_{G_{u}\cap \tilde{G}_{w}}} \nonumber \\
        &=
        \sum_{\ell=0}^{N-k}
        \sum_{\substack{\abs{u}=\abs{w}=N-k, \\ \abs{u\wedge w}=\ell}} \EX{e^{\beta (X^{(k)}_{u}+X_{w}^{(k)})} \Ind_{G_{u}\cap G_{w}}}. \label{eq:expansion}
    \end{align}
    For all $\ell\in\llbracket 0,N-k-1\rrbracket$, we have 
    \begin{align*}
        (X^{(k)}_{u_1},X^{(k)}_{u_2})_{u_1\leq u,w\leq u_2}
        \stackrel{(d)}{=}
        (X^{(k)}_{\ell,n_1},\tilde{X}^{(k)}_{\ell,n_2})_{n_1,n_2\in \llbracket 0,N-k\rrbracket},
    \end{align*}
    and for $\ell= N-k$, 
    \begin{align*}
    (X^{(k)}_{u_1})_{u_1\leq w}= (X^{(k)}_{u_2})_{u_2\leq w} \stackrel{(d)}{=}  (X^{(k)}_n)_{n\in \llbracket 0,N-k\rrbracket},
    \end{align*}
    where $(X^{(k)}_n)_{n\in \llbracket 0,N-k\rrbracket}$ is introduced in Lemma~\ref{lem:many-to-one}. On the other hand, 
    \begin{align*}
        &\text{the number of pairs $(u,w)$ with $\abs{u}=\abs{w}=N-k$ and $\abs{u\wedge w}=\ell$} \nonumber \\
        &=
        \begin{cases}
            2^{2(N-k)-\ell-1}, & \ell\in\llbracket 0,N-k-1\rrbracket, \\
            2^{N-k}, & \ell=N-k.
        \end{cases}
    \end{align*}
    Therefore, \eqref{eq:expansion} yields
    \begin{align}
        &\EX{(Z^{(k),\leq}_{\beta,N-k})^2} \nonumber \\
        &= 2^{N-k} \EX{e^{\beta X_{N-k}^{(k)}} \Ind_{G_{k}}} 
        + \sum_{\ell=0}^{N-k-1} 2^{2(N-k)-\ell-1} \EX{e^{\beta (X^{(k)}_{\ell,N-k}+\tilde{X}_{\ell,N-k}^{(k)})} \Ind_{G_{\ell,k}\cap \tilde{G}_{\ell,k}}}, \nonumber
    \end{align}
    and the proof is completed.
\end{proof}

\subsection{Two useful lemmas for the moment estimates} 
\label{chap3.sec:useful}

The following lemma estimates the difference between $A(y)$ and $A(z)$ when $y$ and $z$ are close to $0$.

\begin{lemma}
\label{chap3.lem:near zero}
For all $y,z\in [0,x_1]$, where $x_1$ is given in (ii) of Assumption~\ref{chap3.assumption}, there exists a constant $C=C(A,x_1)>0$ such that,
\begin{align}
    \abs{A\left(y\right) - A\left( z\right)} \leq \hat{A}'(0) \abs{y-z} + Cy^\alpha\abs{y-z} + C\abs{y-z}^{1+\alpha}, \label{chap3.eq:fine bound}
\end{align}
where $\alpha$ appeared in (ii) of Assumption~\ref{chap3.assumption}.
\end{lemma}
\begin{proof}
Fix $y,z\in [0,x_1]$, where $x_1$ is given in (ii) of Assumption \ref{chap3.assumption}. Then by applying the $1$st order Taylor expansion of $A$ at $y$ with Lagrange remainder, we have 
\begin{align*}
A(z) = A(y) + A'(y)(z-y) + (z-y)(A'(\xi)-A'(y)),
\end{align*}
where $\xi$ is between $z$ and $y$. Then, the triangle inequality and the fact that $A$ is non-decreasing yield
\begin{align}
\abs{A(z)-A(y)}\leq A'(y)\abs{z-y} + \abs{z-y}\abs{A'(\xi)-A'(y)}. \label{chap3.eq:Taylor}
\end{align}
By the local H\"{o}lder continuity of $A$ in a neighborhood around $0$ provided by Assumption~\ref{chap3.assumption}, there exists $C=C(A,x_1)>0$ such that
\begin{align}
A'(y)\leq A'(0) + Cy^\alpha. \label{chap3.eq:Taylor 1st term}
\end{align}
Similarly, by the local H\"{o}lder continuity of $A$ in a neighborhood around $0$ provided by Assumption~\ref{chap3.assumption}, we have
\begin{align}
\abs{A'(\xi)-A'(y)}\leq C\abs{z-y}^\alpha, \label{chap3.eq:Taylor 2nd term}
\end{align}
where $C$ is the same constant as in \eqref{chap3.eq:Taylor 1st term}. Combining \eqref{chap3.eq:Taylor}, \eqref{chap3.eq:Taylor 1st term} and \eqref{chap3.eq:Taylor 2nd term}, we conclude that
\begin{align*}
\abs{A(z)-A(y)}\leq \hat{A}'(0)\abs{z-y} + Cy^\alpha\abs{z-y} + C\abs{z-y}^{1+\alpha},
\end{align*}
and this completes the proof.
\end{proof}

We state and prove the following exponential tilting formula that we apply in the paper.

\begin{lemma}[Exponential tilting]
    \label{lem:exponential tilting}
    Fix $m\in\N$. Let $(Y_n)_{n\in\llbracket 1,m\rrbracket}$ be a collection of independent centered Gaussian with variance $\sigma_n^2 > 0$. For all $n\in\llbracket 1,m\rrbracket$, define $S_n = Y_1+\cdots +Y_n$. Let $g:\llbracket 1,m\rrbracket\rightarrow \R$ be a function and 
    \begin{align*}
        G_n = \{\forall \ell\in\llbracket 1,n\rrbracket : S_\ell \leq g(\ell)\}, \quad n\in\llbracket 1,m\rrbracket.
    \end{align*}
    Then for all $\beta\in\R$,
    \begin{align*}
        \EX{e^{\beta S_n} \Ind_{G_n}}
        = \EX{e^{\beta S_n}}
        \PR{\forall \ell\in\llbracket 1,n\rrbracket : S_\ell \leq g(\ell)-\beta \sum_{r=1}^\ell \sigma_r^2}.
    \end{align*}
\end{lemma}
\begin{proof}
    Fix $n\in\N$. Firstly, from the independence of $\{Y_i\}_{i=1}^n$, we have 
    \begin{align}
        \EX{e^{\beta S_n}}
        = 
        \prod_{i=1}^n \EX{e^{\beta Y_i}}
        = 
        \prod_{i=1}^n \exp(\frac{\beta^2}{2}\sigma_i^2)
        =
        \exp(\frac{\beta}{2}\sum_{i=1}^n \sigma_i^2). \label{eq:product of laplace}
    \end{align}
    Next, for all $x_1,\ldots,x_n\in\R$, we define 
    \begin{align}
        G_n(x_1,\ldots,x_n) 
        &= \{\forall \ell\in\llbracket 1,n\rrbracket : x_1+\cdots+x_\ell \leq g(\ell)\}, \nonumber \\
        \tilde{G}_n(x_1,\ldots,x_n) 
        &= 
        \Big\{\forall \ell\in\llbracket 1,n\rrbracket : x_1+\cdots+x_\ell \leq g(\ell) - \beta \sum_{r=1}^\ell \sigma_r^2
        \Big\}. \nonumber 
    \end{align}
    Then, we have
    \begin{align}
        &\EX{e^{\beta S_n} \Ind_{G_n}}
        \nonumber \\
        &=
        \frac{1}{\prod_{i=1}^n \sqrt{2\pi \sigma_i^2}}
        \int_{\R^n} e^{\beta \sum_{i=1}^n x_i} \Ind_{G_n(x_1,\ldots,x_n)}
        \exp(-\sum_{i=1}^n \frac{x^2_i}{2\sigma_i^2}) \dd{x_1}\cdots\dd{x_n} \nonumber \\
        &=
        \exp(\frac{\beta^2}{2}\sum_{i=1}^n \sigma_i^2) \nonumber \\
        &\quad\quad
        \frac{1}{\prod_{i=1}^n \sqrt{2\pi \sigma_i^2}}
        \int_{\R^n} \Ind_{G_n(x_1,\ldots,x_n)}
        \exp(-\sum_{i=1}^n \frac{(x_i-\beta\sigma_i^2)^2}{2\sigma_i^2}) \dd{x_1}\cdots\dd{x_n} \nonumber \\
        &=
        \exp(\frac{\beta^2}{2}\sum_{i=1}^n \sigma_i^2)
        \frac{1}{\prod_{i=1}^n \sqrt{2\pi \sigma_i^2}}
        \int_{\R^n} \Ind_{\tilde{G}_n(y_1,\ldots,y_n)}
        \exp(-\sum_{i=1}^n \frac{y_i^2}{2\sigma_i^2}) \dd{y_1}\cdots\dd{y_n} \nonumber \\
        &= \EX{e^{\beta S_n}} \PR{\forall \ell\in\llbracket 1,n\rrbracket : S_\ell \leq g(\ell)-\beta \sum_{r=1}^\ell \sigma_r^2}, \nonumber 
    \end{align}
    where the second to last line follows from the change of variables $y_i=x_i - \beta \sigma_i^2$, $i=1,\ldots,n$, and the last line follows from \eqref{eq:product of laplace}.
\end{proof}

\subsection{Moment estimates of the truncated partition function} \label{chap3.sec:moment estimates}

We now proceed with the first moment estimate.

\begin{lemma}[First moment estimate]
\label{chap3.lem:first moment}
For all 
\begin{align*}
a>0 
\quad \text{and} \quad
b\geq \frac{\hat{A}'(0)}{a}\log(10\max\left\{\frac{e^{-a^2/(2\hat{A}'(0))}}{1-e^{-a^2/(2\hat{A}'(0))}},1\right\}), 
\end{align*}
there exist $N_0=N_0(A,\beta,a,b)\in\N$ such that for all $N\geq N_0$ and $k\in\llbracket 0,K\rrbracket$ where $K=\floor{2\log_\gamma N}$ with $\gamma\in(11/10,2)$,
\begin{align*}
\EX{Z_{\beta,N-k}^{(k),\leq}} \geq \frac{7}{10} \EX{Z_{\beta,N-k}^{(k)}}.
\end{align*}
\end{lemma}
\begin{proof}
By Lemma~\ref{lem:many-to-one}, we have
\begin{align}
\EX{Z^{(k),\leq}_{\beta,N-k}} 
&= 2^{N-k}\EX{e^{\beta X^{(k)}_{N-k}} \Ind_{G_{N-k}}}. \label{eq:many-to-one}
\end{align}
By Lemma~\ref{lem:exponential tilting}, we have 
\begin{align}
&\EX{e^{\beta X^{(k)}_{N-k}} \Ind_{G_{N-k}}} \nonumber \\
&= 
\exp(\frac{\beta^2 N}{2}\left(1-A\left(\frac{k}{N}\right)\right)) \cdot \PR{\forall n\in \llbracket 1,N-k \rrbracket : X^{(k)}_{N-k}(n)\leq a\cdot n + b }. \label{chap3.eq:Z tr m1.1}
\end{align}
We want to give a lower bound of the probability in \eqref{chap3.eq:Z tr m1.1}. The union bound and the Chernoff bound yield 
\begin{align}
\PR{\exists n\in \llbracket 1,N-k \rrbracket : X^{(k)}_{N-k}(n)> a\cdot n + b }
&\leq \sum_{n=1}^{N-k} \PR{X^{(k)}_{N-k}(n)> a\cdot n + b } \nonumber \\
&\leq \sum_{n=1}^{N-k} \exp\left(-\frac{(a\cdot n+b)^2}{2N\cdot (A(\frac{n+k}{N})-A(\frac{k}{N}))}\right). \label{chap3.eq:prob.0.0}
\end{align}
We separate \eqref{chap3.eq:prob.0.0} into two terms
\begin{align}
\eqref{chap3.eq:prob.0.0}
= \sum_{n=1}^{K} \exp\left(-\frac{(a\cdot n+b)^2}{2N\cdot (A(\frac{n+k}{N})-A(\frac{k}{N}))}\right) 
+ \sum_{n=K+1}^{N-k} \exp\left(-\frac{(a\cdot n+b)^2}{2N\cdot (A(\frac{n+k}{N})-A(\frac{k}{N}))}\right).
\label{chap3.eq:prob.0}
\end{align}

We start with bounding the first term of \eqref{chap3.eq:prob.0}. Take $N_1\in\N$ such that $2K/N\leq x_1$ for all $N\geq N_1$. Then by Lemma~\ref{chap3.lem:near zero},
for all $n\in \llbracket 1,K\rrbracket$, we have
\begin{align}
N\cdot \left(A\left(\frac{n+k}{N}\right)-A\left(\frac{k}{N}\right)\right)
&\leq \hat{A}'(0) n + C\frac{n(n+k)^\alpha}{N} + C\frac{n^{1+\alpha}}{N^\alpha} \nonumber \\
&\leq \hat{A}'(0) n + C\frac{K(2K)^\alpha}{N} + C\frac{K^{1+\alpha}}{N^{\alpha}} \label{chap3.eq:diff of A upp.0} \\
&= \hat{A}'(0) n + h(N),
\label{chap3.eq:diff of A upp} \\
&\leq (\hat{A}'(0)+h(N))n. \label{chap3.eq:diff of A upp.01}
\end{align}
where $h(N)$ is defined as
\begin{align}
    h(N)= C(1+2^{\alpha}) \frac{K^{1+\alpha}}{N}, \label{eq:h(N)}
\end{align}
\eqref{chap3.eq:diff of A upp.0} follows by the assumption that $n\leq K$ and $k\leq K$ and \eqref{chap3.eq:diff of A upp.01} follows from the fact that $n\geq 1$ and $h(N)\geq 0$.

By \eqref{chap3.eq:diff of A upp.01} and the fact that $(an+b)^2 \geq a^2n^2 + 2abn$ for all $a,b,n>0$, the first term of \eqref{chap3.eq:prob.0} is bounded from above by
\begin{align}
&\sum_{n=1}^{K} \exp\left(-\frac{(a\cdot n+b)^2}{2N\cdot (A(\frac{n+k}{N})-A(\frac{k}{N}))}\right) \nonumber \\
&\leq \sum_{n=1}^{K} \exp\left(-\frac{a^2 n^2 + 2abn}{2(\hat{A}'(0) + h(N))n}\right) \nonumber \\
&=
\exp(-\frac{ab}{\hat{A}'(0)+h(N)})\sum_{n=1}^{K} \exp(-\frac{a^2n}{2(\hat{A}'(0)+h(N))}).
\label{chap3.eq:diff of A upp.1}
\end{align}
The summation in \eqref{chap3.eq:diff of A upp.1} is bounded from above by
\begin{align}
\sum_{n=1}^{K} \exp(-\frac{a^2n}{2(\hat{A}'(0)+h(N))})
\leq \sum_{n=1}^{\infty} \exp(-\frac{a^2n}{2(\hat{A}'(0)+h(N))})
\leq \frac{\exp(-\frac{a^2}{2(\hat{A}'(0)+h(N))})}{1-\exp(-\frac{a^2}{2(\hat{A}'(0)+h(N))})}.\label{chap3.eq:prob first bound}
\end{align}
Thus, \eqref{chap3.eq:diff of A upp.1} and \eqref{chap3.eq:prob first bound} yield
\begin{align}
    &\sum_{n=1}^{K} \exp\left(-\frac{(a\cdot n+b)^2}{2N\cdot (A(\frac{n+k}{N})-A(\frac{k}{N}))}\right) \nonumber \\
    &\leq 
    \exp(-\frac{ab}{\hat{A}'(0)+h(N)})
    \frac{\exp(-\frac{a^2}{2(\hat{A}'(0)+h(N))})}{1-\exp(-\frac{a^2}{2(\hat{A}'(0)+h(N))})}.
    \label{eq:prob first bound.1}
\end{align}
Since $h(N)\rightarrow 0$ as $N\rightarrow\infty$, by the choice of parameters $a$ and $b$, we have
\begin{align}
\limsup_{N\rightarrow\infty} \frac{\exp(-\frac{a^2}{2(\hat{A}'(0)+h(N))})}{1-\exp(-\frac{a^2}{2(\hat{A}'(0)+h(N))})}\exp(-\frac{ab}{\hat{A}'(0)+h(N)}) 
&= \frac{e^{-a^2/(2\hat{A}'(0))}}{1-e^{-a^2/(2\hat{A}'(0))}}e^{-ab/\hat{A}'(0)} \nonumber \\
&\leq \frac{1}{10}. \label{chap3.eq:limit}
\end{align}
Combining \eqref{eq:prob first bound.1} and \eqref{chap3.eq:limit}, there exists $N_2=N_2(A,\beta)\in\N$ such that for all $N\in N_2$, the first term of \eqref{chap3.eq:prob.0} is bounded from above by 
\begin{align}
\sum_{n=1}^{K} \exp\left(-\frac{(a\cdot n+b)^2}{2N\cdot (A(\frac{n+k}{N})-A(\frac{k}{N}))}\right) \leq \frac{2}{10}. \label{chap3.eq:first bound}
\end{align}

It remains to bound the second term of \eqref{chap3.eq:prob.0}. By Assumption~\ref{chap3.assumption}, we know that $A(x)\geq 0$ for all $x\in [0,1]$ and $A(x)\leq \hat{A}(0)x$ for all $x\in [0,1]$. Thus,
\begin{align}
    N\cdot \left(A\left(\frac{n+k}{N}\right)-A\left(\frac{k}{N}\right) 
    \right)
    \leq N A\left(\frac{n+k}{N}\right)
    \leq \hat{A}'(0)(n+k). \label{eq:second difference bound}
\end{align}
For $n\geq K+1$ and $k\leq K$, we have 
\begin{align}
    1+\frac{k}{n} \leq 1+\frac{K}{K+1} \leq 2. \label{eq:bound by 2}
\end{align}
Combining \eqref{eq:second difference bound} and \eqref{eq:bound by 2} then yield for $n\geq K+1$ and $k\leq K$,
\begin{align}
    N\cdot \left(A\left(\frac{n+k}{N}\right)-A\left(\frac{k}{N}\right) 
    \right)
    \leq 2\hat{A}'(0) n. \label{eq:second difference bound.1}
\end{align}
Thus, by \eqref{eq:second difference bound.1} and the fact that $(an+b)^2 \geq a^2 n^2$ for all $a,b,n>0$,
\begin{align}
\sum_{n=K+1}^{N-k} \exp\left(-\frac{(a n+b)^2}{2N\cdot (A(\frac{n+k}{N})-A(\frac{k}{N}))}\right)
&\leq \sum_{n=K+1}^{N-k} \exp\left(-\frac{a^2n^2}{4\hat{A}(0) n}\right) \nonumber \\
&=
\sum_{n=K+1}^{N-k} \exp\left(-\frac{a^2}{4\hat{A}'(0)}n\right) \nonumber \\
&\leq 
\sum_{n=K+1}^{\infty} \exp\left(-\frac{a^2}{4\hat{A}'(0)}n\right) \nonumber \\
&\leq 
\exp(-\frac{a^2}{4\hat{A}'(0)}(K+1))\frac{1}{1-\exp(-\frac{a^2}{4\hat{A}'(0)})}. \label{chap3.eq:second bound}
\end{align}
By our choice of $K$, the limit superior of \eqref{chap3.eq:second bound} equals $0$ as $N\rightarrow\infty$. In particular, there exists $N_3=N_3(A,\beta)$ such that for all $N\geq N_2$, 
\begin{align}
&\sum_{n=K+1}^{N-k} \exp\left(-\frac{(a\cdot n+b)^2}{2N\cdot (A(\frac{n+k}{N})-A(\frac{k}{N}))}\right)
\leq \frac{1}{10}. \label{chap3.eq:second bound.1}
\end{align}
Let $N_0\coloneqq \max\{N_1,N_2,N_3\}$. Combining \eqref{chap3.eq:prob.0}, \eqref{chap3.eq:first bound} and \eqref{chap3.eq:second bound.1}, for all $N\geq N_0$ and for all $k\in\llbracket 0,K\rrbracket$, 
\begin{align}
\PR{\exists n\in \llbracket 1,N-k \rrbracket : X^{(k)}_{N-k}(n)> a\cdot n + b }
\leq \frac{3}{10}. \label{chap3.eq:3/10}
\end{align}
Combining \eqref{eq:expectation of part}, \eqref{chap3.eq:Z tr m1.1}, \eqref{chap3.eq:Z tr m1.1} and \eqref{chap3.eq:3/10}, we conclude that for all $N\geq N_0$ and for all $k\in\llbracket 0,K\rrbracket$,
\begin{align}
\EX{Z^{(k),\leq}_{\beta,N-k}} 
\geq \frac{7}{10} 2^{N-k}\exp(\frac{\beta^2 N}{2}\left(1-A\left(\frac{k}{N}\right)\right))
=
\EX{Z^{(k)}_{\beta,N-k}},\nonumber 
\end{align}
where the equality above follows from \eqref{eq:expectation of part}.
\end{proof}

It remains to provide a second moment estimate of the truncated partition function.

\begin{lemma}[Second moment estimate]
\label{chap3.lem:second moment}
Let $\beta<\beta_c$. 
For all 
\begin{align*}
    a\in \left(0,\frac{\log 2}{\beta}-\frac{1}{2}\beta \hat{A}'(0)\right)
    \quad \text{and} \quad 
    b>0, 
\end{align*}
there exist $N_0=N_0(A,\beta)\in\N$ and $C=C(A,\beta)>0$ such that for all $N\geq N_0$ and $k\in\llbracket 0,K \rrbracket$, where 
$K=\floor{2\log_\gamma N}$ with $\gamma\in(11/10,2)$,
we have
\begin{align*}
\EX{(Z^{(k),\leq}_{\beta,N-k})^2}
\leq C\EX{Z^{(k)}_{\beta,N-k}}^2. 
\end{align*}
\end{lemma}
\begin{proof}
Fix $\beta<\beta_c$ and $a\in (0,\frac{\log 2}{\beta}-\frac{1}{2}\beta \hat{A}'(0))$. Define
\begin{align}
c=c(A,\beta)\coloneqq \log 2 - \frac{1}{2}\beta^2\hat{A}'(0)-a >0,\label{chap3.eq:small c}
\end{align}
where the positivity follows from the definition of $\beta_c$ in \eqref{chap3.eq:threshold}.

By Lemma~\ref{lem:many-to-two}, we have
\begin{align}
&\EX{(Z^{(k),\leq}_{\beta,N-k})^2} \nonumber \\
        &= 2^{N-k} \EX{e^{2\beta X_{N-k}^{(k)}} \Ind_{G_{k}}} 
        + \sum_{\ell=0}^{N-k-1} 2^{2(N-k)-\ell-1} \EX{e^{\beta (X^{(k)}_{\ell,N-k}+\tilde{X}_{\ell,N-k}^{(k)})} \Ind_{G_{\ell,k}\cap \tilde{G}_{\ell,k}}}. \label{chap3.eq:many-to-two}
\end{align}
We now want to estimate the expectations in \eqref{chap3.eq:many-to-two}. For all $\ell\in\llbracket 1,N\rrbracket$ and $n\in\llbracket 1, N-\ell\rrbracket$, define 
\begin{align*}
    H_n^{(\ell)} = \left\{X_{n}^{(\ell)}\leq \beta N \left(A\left(\frac{\ell+n}{N}\right)-A\left(\frac{\ell}{N}\right)\right) + a n + b \right\}.
\end{align*}
Then, we see that the expectation in the first term of \eqref{chap3.eq:many-to-two} is bounded from above by
\begin{align}
    \EX{e^{2\beta X_{N-k}^{(k)}} \Ind_{G_{k}}}
    \leq \EX{e^{2\beta X^{(k)}_{N-k}}\Ind_{H_{N-k}^{(k)}}}.
\end{align}
Fix $\ell\in\llbracket 0,N-k-1\rrbracket$.  We have
\begin{align}
G_{\ell,k}\cap \tilde{G}_{\ell,k} \subseteq \left\{X_{\ell,\ell}^{(k)}\leq \beta N \left(A\left(\frac{\ell+k}{N}\right)-A\left(\frac{k}{N}\right)\right) + a \ell + b \right\} = H_{\ell,\ell}^{(k)}. \label{chap3.eq:coarser truncation}
\end{align}
Then, we have
\begin{align}
\EX{e^{\beta (X^{(k)}_{\ell,N-k}+\tilde{X}_{\ell,N-k}^{(k)})} \Ind_{G_{\ell,k}\cap \tilde{G}_{\ell,k}}} 
&\leq 
\EX{e^{\beta (X^{(k)}_{\ell,N-k}+\tilde{X}_{\ell,N-k}^{(k)})} \Ind_{H_{\ell,\ell}^{(k)}}
} \label{eq:intersection bound} \\
&= \EX{e^{2\beta X_{\ell,\ell}^{(k)}} \Ind_{H_{\ell,\ell}^{(k)}}
} 
\EX{e^{\beta (X^{(k)}_{\ell,N-k}+\tilde{X}^{(k)}_{\ell,N-k}-2X^{(k)}_{\ell,\ell})}} \label{eq:factorization} \\
&= \EX{e^{2\beta X_{\ell}^{(k)}} 
\Ind_{H_{\ell}^{(k)}
}
}
\EX{e^{\beta X^{(\ell+k)}_{N-\ell-k}}}^2 \label{chap3.eq:two terms.2.1} \\
&=
\EX{e^{2\beta X_{\ell}^{(k)}} 
\Ind_{H_{\ell}^{(k)}
}
}
\exp(\beta^2 N \left(1-A\left(\frac{\ell+k}{N}\right)\right)),
\label{chap3.eq:two terms.2}
\end{align}
where \eqref{eq:intersection bound} follows from \eqref{chap3.eq:coarser truncation}, and both \eqref{eq:factorization} and \eqref{chap3.eq:two terms.2.1} from Lemma~\ref{lem:many-to-two}.

For $\ell\in\llbracket 0,N-k\rrbracket$, Lemma~\ref{lem:exponential tilting} yields
\begin{align}
    \EX{e^{2\beta X_{\ell}^{(k)}} \Ind_{H_{\ell}^{(k)}}}
    &=
    \exp(2\beta^2 N \left(A\left(\frac{\ell+k}{N}\right)-A\left(\frac{k}{N}\right)\right)) \nonumber \\
    &\quad\quad\quad
    \PR{
    X_{\ell}^{(k)}\leq -\beta N \left(A\left(\frac{k+\ell}{N}\right)-A\left(\frac{k}{N}\right)\right) + a \ell + b
    }. \label{eq:H bound}
\end{align}
We distinguish the following two cases.

\noindent \textbf{Case 1: $-\beta N \left(A\left(\frac{\ell+k}{N}\right)-A\left(\frac{k}{N}\right)\right) + a \ell + b\geq 0$.} In this case, by bounding the probability in \eqref{eq:H bound} by $1$, we obtain
\begin{align}
\eqref{eq:H bound} 
&\leq \exp(2\beta^2 N \left(A\left(\frac{\ell+k}{N}\right)-A\left(\frac{k}{N}\right)\right)) \nonumber \\
&\leq \exp(\beta^2 N \left(A\left(\frac{\ell+k}{N}\right)-A\left(\frac{k}{N}\right)\right))\exp(\beta(a\ell+b)). \label{chap3.eq:case distinction.1}
\end{align}

\noindent \textbf{Case 2: $-\beta N \left(A\left(\frac{\ell+k}{N}\right)-A\left(\frac{k}{N}\right)\right) + a \ell + b<0$.} In this case, the Chernoff bound yields
\begin{align}
\eqref{eq:H bound}
&\leq \exp(2\beta^2 N \left(A\left(\frac{\ell+k}{N}\right)-A\left(\frac{k}{N}\right)\right)) \nonumber \\
&\quad\quad
\cdot \exp( -\frac{\left(\beta N (A\left(\frac{\ell+k}{N}\right)-A\left(\frac{k}{N}\right)) - (a \ell+b)\right)^2}{2N \left(A\left(\frac{\ell+k}{N}\right)-A\left(\frac{k}{N}\right)\right)} ). \label{eq:H bound.1}
\end{align}
By the fact that 
\begin{align*}
    &\left(\beta N \left(A\left(\frac{\ell+k}{N}\right)-A\left(\frac{k}{N}\right)\right) - (a \ell+b)\right)^2 \nonumber \\
    &\geq
    \beta^2 N^2\left(A\left(\frac{\ell+k}{N}\right)-A\left(\frac{k}{N}\right)\right)^2 - 2 \beta N \left(A\left(\frac{\ell+k}{N}\right)-A\left(\frac{k}{N}\right)\right)(a \ell+b),
\end{align*}
\eqref{eq:H bound.1} yields the bound
\begin{align}
\eqref{eq:H bound}
&\leq \exp(\frac{3}{2}\beta^2 N \left(A\left(\frac{\ell+k}{N}\right)-A\left(\frac{k}{N}\right)\right)) \exp(\beta (a \ell + b)). \label{chap3.eq:1st term}
\end{align}
From the case study above, we conclude that 
\begin{align}
    \EX{e^{2\beta X_{\ell}^{(k)}} \Ind_{H_{\ell}^{(k)}}}
    \leq \exp(\frac{3}{2}\beta^2 N \left(A\left(\frac{\ell+k}{N}\right)-A\left(\frac{k}{N}\right)\right)) \exp(\beta (a \ell + b)). \label{eq:case study}
\end{align}

Applying the bounds \eqref{chap3.eq:two terms.2} and \eqref{eq:case study} to \eqref{chap3.eq:many-to-two}, we obtain 
\begin{align}
    &\EX{(Z^{(k),\leq}_{\beta,N-k})^2} \nonumber \\
    &\leq 
    2^{N-k} \exp(\frac{3}{2}\beta^2 N \left(A\left(\frac{1}{N}\right)-A\left(\frac{k}{N}\right)\right)) \exp(\beta (a (N-k) + b)) \nonumber \\
    &+ \sum_{\ell=0}^{N-k-1} 2^{2(N-k)-\ell-1} \exp(\frac{3}{2}\beta^2 N \left(A\left(\frac{\ell+k}{N}\right)-A\left(\frac{k}{N}\right)\right)) \exp(\beta (a \ell + b))
    \exp(\beta^2 N \left(1-A\left(\frac{\ell+k}{N}\right)\right)) \nonumber \\
    &\leq 
    2^{2(N-k)}\exp(\beta^2 N \left(1-A\left(\frac{k}{N}\right)\right)) \nonumber \\
    &\quad\quad
    \sum_{\ell=0}^{N-k} 2^{-\ell} 
    \exp(\frac{1}{2}\beta^2 N \left(A\left(\frac{\ell+k}{N}\right)-A\left(\frac{k}{N}\right)\right)) \exp(\beta (a \ell + b)).
    \label{eq:truncated bound result.0} \\
    &=
    \EX{Z_{\beta,N-k}^{(k)}}^2  \sum_{\ell=0}^{N-k} 2^{-\ell} 
    \exp(\frac{1}{2}\beta^2 N \left(A\left(\frac{\ell+k}{N}\right)-A\left(\frac{k}{N}\right)\right)) \exp(\beta (a \ell + b)), \label{eq:truncated bound result.1}
\end{align}
where \eqref{eq:truncated bound result.0} follows from \eqref{eq:expectation of part} and the fact that $2^{-\ell-1}\leq 2^{-\ell}$ for any integer $\ell$. Now, we decompose the summation in \eqref{eq:truncated bound result.1} into
\begin{align}
    &\sum_{\ell=0}^{N-k} 2^{-\ell} 
    \exp(\frac{1}{2}\beta^2 N \left(A\left(\frac{\ell+k}{N}\right)-A\left(\frac{k}{N}\right)\right)) \exp(\beta (a \ell + b)) \nonumber \\
    &= 
    \sum_{\ell=0}^{K_1} 2^{-\ell} 
    \exp(\frac{1}{2}\beta^2 N \left(A\left(\frac{\ell+k}{N}\right)-A\left(\frac{k}{N}\right)\right)) \exp(\beta (a \ell + b)) \nonumber \\
    &+ 
    \sum_{\ell=K_1+1}^{N-k} 2^{-\ell} 
    \exp(\frac{1}{2}\beta^2 N \left(A\left(\frac{\ell+k}{N}\right)-A\left(\frac{k}{N}\right)\right)) \exp(\beta (a \ell + b)),\label{eq:summation decomposition}
\end{align}
where $K_1=K^2$.
We start with bounding the first term in \eqref{eq:summation decomposition}. Take $N_1$ such that $(K_1+K)/N\leq x_1$ for all $N\geq N_1$.
Following the same derivation for \eqref{chap3.eq:diff of A upp}, for all $N\geq N_1$, $k\in\llbracket 1,K\rrbracket$ and $\ell\in\llbracket 0,K_1\rrbracket$, we have
\begin{align}
    N \left(A\left(\frac{\ell+k}{N}\right)-A\left(\frac{k}{N}\right)\right) \leq \hat{A}'(0)\ell + h(N), \label{eq:difference bound.2}
\end{align}
where $h(N)$ is defined in \eqref{eq:h(N)}.  
Applying \eqref{eq:difference bound.2} then yields
\begin{align}
    \text{first term in \eqref{eq:summation decomposition}}
    &\leq \exp(h(N)+\beta b)\sum_{\ell=0}^{K_1} \exp(-(\log 2 - \frac{\beta^2}{2}\hat{A}(0)-\beta a)\ell) \nonumber \\
    &= \exp(h(N)+\beta b)\sum_{\ell=0}^{K_1} \exp(-c\ell) \nonumber \\
    &\leq \exp(h(N)+\beta b)\sum_{\ell=0}^{\infty} \exp(-c\ell) \nonumber \\
    &= \exp(h(N)+\beta b) \frac{1}{1-e^{-c}}. \label{eq:first term bound}
\end{align}
It remains to bound the second term in \eqref{eq:summation decomposition}. As in \eqref{eq:second difference bound}, Assumption~\ref{chap3.assumption} implies that $A(x)\geq 0$ for all $x\in [0,1]$ and $A(x)\leq \hat{A}(0)x$ for all $x\in [0,1]$, and so
\begin{align}
    N\cdot \left(A\left(\frac{\ell+k}{N}\right)-A\left(\frac{k}{N}\right) 
    \right)
    \leq \hat{A}'(0)(\ell+k) \leq \hat{A}'(0)\ell + \hat{A}'(0)K,
\end{align}
where the second inequality above is by the assumption that $k\leq K$. Thus, 
\begin{align}
    \text{second term in \eqref{eq:summation decomposition}}
    &\leq \exp(\hat{A}'(0)K+\beta b)
    \sum_{\ell=K_1+1}^{N-k} \exp(-(\log 2 - \frac{\beta^2}{2}\hat{A}'(0)-\beta a)\ell) \nonumber \\
    &= \exp(\hat{A}'(0)K+\beta b)
    \sum_{\ell=K_1+1}^{N-k} \exp(-c\ell) \nonumber \\
    &\leq \exp(\hat{A}'(0)K+\beta b)
    \sum_{\ell=K_1+1}^{\infty} \exp(-c\ell) \nonumber \\
    &= \exp(\hat{A}'(0)K+\beta b)\frac{e^{-c(K_1+1)}}{1-e^{-c}} \nonumber \\
    &= \exp(-L(N)+\beta b) \frac{e^{-c}}{1-e^{-c}}, \label{eq:second term bound}
\end{align}
where $L(N) = c K_1 - \hat{A}'(0)K = c K^2 - \hat{A}'(0)K$. Combining \eqref{eq:summation decomposition}, \eqref{eq:first term bound} and \eqref{eq:second term bound}, we derive that
\begin{align}
    &\sum_{\ell=0}^{N-k} 2^{-\ell} 
    \exp(\frac{1}{2}\beta^2 N \left(A\left(\frac{\ell+k}{N}\right)-A\left(\frac{k}{N}\right)\right)) \exp(\beta (a \ell + b)) \nonumber \\
    &\leq 
    \exp(h(N)+\beta b) \frac{1}{1-e^{-c}} + \exp(-L(N)+\beta b) \frac{e^{-c}}{1-e^{-c}}. \label{eq:bound conclusion}
\end{align}
Note that $h(N)\rightarrow 0$ and $L(N)\rightarrow\infty$ as $N\rightarrow\infty$, so there exists $N_2$ such that $h(N)\leq 1$ and $L(N)\geq 1$ for all $N\geq N_2$. Take $N_0=\max\{N_1,N_2\}$. For all $N\geq N_0$,
we conclude from \eqref{eq:truncated bound result.1} and \eqref{eq:bound conclusion}
\begin{align*}
    \EX{(Z^{(k),\leq}_{\beta,N-k})^2}
    \leq C \EX{Z^{(k)}_{\beta,N-k}}^2
\end{align*}
where $C = C(A,\beta) = \exp(1+\beta b) \frac{1}{1-e^{-c}} + \exp(-1+\beta b) \frac{e^{-c}}{1-e^{-c}}$.
\end{proof}

\subsection{Proof of Proposition~\ref{chap3.prop:init}} \label{chap3.sec:proof of prop:init}

We are now ready to prove Proposition~\ref{chap3.prop:init}.

\begin{proof}[Proof of Proposition~\ref{chap3.prop:init}]
Fix $\beta<\beta_c$. By Lemma~\ref{chap3.lem:first moment} and Lemma~\ref{chap3.lem:second moment}, if we choose
\begin{align}
a=\frac{1}{2\beta}(\log 2 -\frac{1}{2}\beta^2\hat{A}'(0)) \quad \text{and} \quad b = \frac{\hat{A}'(0)}{a}\log(10\max\left\{\frac{e^{-a^2/(2\hat{A}'(0))}}{1-e^{-a^2/(2\hat{A}'(0))}},1\right\}), \label{chap3.eq:choices of a and b}
\end{align}
then there exist $N_1=N_1(A,\beta)\in\N$, $N_2=N_2(A,\beta)\in\N$ and $C=C(A,\beta)>0$ such that for all $N\geq N_1$ and $k\in\llbracket 0,K\rrbracket$,
\begin{align}
\EX{Z_{\beta,N-k}^{(k),\leq}} \geq \frac{7}{10}\EX{Z_{\beta,N-k}^{(k)}}, \label{chap3.eq:first moment conclusion}
\end{align}
and for all $N\geq N_2$ and $k\in\llbracket 0,K\rrbracket$,
\begin{align}
\EX{(Z_{\beta,N-k}^{(k),\leq})^2} \leq C \EX{Z_{\beta,N-k}^{(k)}}^2. \label{chap3.eq:second moment conclusion}
\end{align}

Now, since $Z^{(k),\leq}_{\beta,N-k} \leq Z^{(k)}_{\beta,N-k}$, we have 
\begin{align}
\PR{Z^{(k)}_{\beta,N-k} > \frac{1}{2} \EX{Z^{(k)}_{\beta,N-k}}}
&\geq \PR{Z^{(k),\leq}_{\beta,N-k} > \frac{1}{2}  \EX{Z^{(k)}_{\beta,N-k}}}. \label{chap3.eq:trivial bound for truncated partition function}
\end{align}
On the other hand, by the Cauchy--Schwarz inequality, we have
\begin{align}
&\EX{Z^{(k),\leq}_{\beta,N-k}} \nonumber \\
&= \EX{Z^{(k),\leq}_{\beta,N-k}\Ind\left\{Z^{(k),\leq}_{\beta,N-k} > \frac{1}{2}\EX{Z^{(k)}_{\beta,N-k}}\right\}}
+ \underbrace{\EX{Z^{(k),\leq}_{\beta,N-k}\Ind\left\{Z^{(k),\leq}_{\beta,N-k} \leq \frac{1}{2}\EX{Z^{(k)}_{\beta,N-k}}\right\}}}_{\leq \frac{1}{2}\EX{Z^{(k)}_{\beta,N-k}}} \nonumber \\
&\leq \EX{(Z^{(k),\leq}_{\beta,N-k})^2}^{1/2}\PR{Z^{(k),\leq}_{\beta,N-k} > \frac{1}{2} \EX{Z^{(k)}_{\beta,N-k}}}^{1/2} 
+ \frac{1}{2} \EX{Z^{(k)}_{\beta,N-k}}. \label{chap3.eq:Paley-Zygmund}
\end{align}
Let $N_0 = N_0(A,\beta) \coloneqq \max\{N_1,N_2\}$. Combining \eqref{chap3.eq:first moment conclusion}, \eqref{chap3.eq:second moment conclusion} and \eqref{chap3.eq:Paley-Zygmund}, we derive that for all $N\geq N_0$ and $k\in\llbracket 0,K\rrbracket$, 
\begin{align}
\PR{Z^{(k),\leq}_{\beta,N-k}\geq \frac{1}{2} \EX{Z^{(k)}_{\beta,N-k}}}^{1/2}
\geq \frac{\displaystyle \EX{Z^{(k),\leq}_{\beta,N-k}}-\frac{1}{2} \EX{Z^{(k)}_{\beta,N-k}}}{\EX{(Z^{(k),\leq}_{\beta,N-k})^2}^{1/2}}
\geq \frac{2}{10\sqrt{C}}. \label{chap3.eq:lower bound}
\end{align}
Combining \eqref{chap3.eq:trivial bound for truncated partition function} and \eqref{chap3.eq:lower bound},
we conclude that for all $N\geq N_0$ and $k\in\llbracket 0,K\rrbracket$, 
\begin{align}
\PR{Z^{(k)}_{\beta,N-k} \leq \frac{1}{2} \EX{Z^{(k)}_{\beta,N-k}}}
&= 1- \PR{Z^{(k)}_{\beta,N-k} > \frac{1}{2} \EX{Z^{(k)}_{\beta,N-k}}} \nonumber \\
&\leq 1- \PR{Z^{(k),\leq}_{\beta,N-k} > \frac{1}{2}\cdot \EX{Z^{(k)}_{\beta,N-k}}} \nonumber \\
&\leq 1- \frac{4}{100C}. \nonumber 
\end{align}
By taking $\eta_0 = \eta_0(A,\beta) = 1-\frac{4}{100 C}$, the proof is completed.
\end{proof}

\section{Left tail estimates via a bootstrap argument} \label{chap3.sec:bootstrap}

The main goal of this section is to provide a finer estimate of the left tail of $Z_{\beta,N}$ than the one provided by Proposition~\ref{chap3.prop:init}. Namely, we want to construct two sequences that satisfy \eqref{chap3.eq:double exponential decay}, and this is manifested in the following proposition.
\begin{proposition}
\label{chap3.prop:bootstrap}
Let $K=\floor{2 \log_\gamma N}$ with $\gamma\in (11/10,2)$. Then, there exist $N_0=N_0(A,\beta)\in\N$ and two sequences $(\eta_k)_{k=0}^K$ and $(\varepsilon_k)_{k=0}^K$ such that
\begin{enumerate}[label = (\roman*)]
\item for all $N\geq N_0$ and $k\in \llbracket 0,K \rrbracket$,
\begin{align*}
    \PR{Z_{\beta,N} \leq \varepsilon_k \EX{Z_{\beta,N}}} \leq \eta_k.
\end{align*}
\item for all $s>0$, there exists a constant $C=C(A,\beta,s)>0$, independent of $K$, such that 
\begin{align*}
\sum_{r=0}^{K-1} (\varepsilon_{r+1})^{-s} \eta_r \leq C.
\end{align*}
\end{enumerate}
\end{proposition}

The proof of Proposition~\ref{chap3.prop:bootstrap} requires the following two lemmas. 
For all $k\in\llbracket 0,N-1\rrbracket$, define
\begin{align*}
M^{(k)}_{\beta,1} = \frac{e^{\beta X^{(k)}_{1}}}{\EX{Z^{(k)}_{\beta,1}}},
\end{align*}
where $X^{(k)}_1$ is defined in Lemma~\ref{lem:many-to-one}.
The first lemma states a bootstrap inequality.
\begin{lemma}
\label{chap3.lem:bootstrap} 
For all $c>0$, $\delta>0$, $N\in\N$, and $k\in \llbracket 0,N-1\rrbracket$, we have
\begin{align*}
\PR{Z^{(k)}_{\beta,N-k} \leq c\delta\cdot \EX{Z^{(k)}_{\beta,N-k}}}
&\leq \left[\PR{Z^{(k+1)}_{\beta,N-k-1} \leq c \cdot \EX{Z^{(k+1)}_{\beta,N-k-1}}} + \PR{M^{(k)}_{\beta,1} \leq \delta} \right]^2.
\end{align*}
\end{lemma}
\begin{proof}
Fix $N\in\N$. For every $c>0$, $\delta>0$ and $k\in \llbracket 0,N-1\rrbracket$, we have 
\begin{align}
&\PR{Z^{(k)}_{\beta,N-k} \leq c\delta\cdot \EX{Z^{(k)}_{\beta,N-k}} } \nonumber \\
&= \PR{\sum_{\abs{u}=1} M_{\beta,u}^{(k)} \cdot Z^{(k),u}_{\beta,N-k-1} \leq c\delta\cdot \EX{Z^{(k+1)}_{\beta,N-k-1}}} \nonumber \\
&\leq \PR{\forall \abs{u}=1 : M_{\beta,u}^{(k)} \cdot Z^{(k),u}_{\beta,N-k-1}  \leq c\delta\cdot \EX{Z^{(k+1)}_{\beta,N-k-1}}} \nonumber \\
&= \PR{M_{\beta,1}^{(k)}\cdot Z^{(k+1)}_{\beta,N-k-1} \leq c\delta\cdot \EX{Z^{(k+1)}_{\beta,N-k-1}}}^2 \label{eq:M independence} \\
&\leq \left[\PR{Z^{(k+1)}_{\beta,N-k-1} \leq c \cdot \EX{Z^{(k+1)}_{\beta,N-k-1}}} + \PR{M_{\beta,1}^{(k)} \leq \delta} \right]^2, \nonumber 
\end{align}
where \eqref{eq:M independence} follows from independence.
\end{proof}

The next lemma provides a uniform estimate for the left tail of $M_{\beta,1}^{(k)}$ when $k$ is small compared to $N$.
\begin{lemma}
\label{chap3.lem:one step}
Let $N\in\N$ and $K=\floor{2 \log_\gamma N}$ with $\gamma\in(11/10,2)$. For all $\beta>0$, $k\in\llbracket 0, K\rrbracket$ and $s>0$, there exist $N_0=N(A,\beta,s)\in\N$ and a constant $C=C(A,\beta,s)\geq 1$, independent of $N$, such that for all $N\geq N_0$,
\begin{align*}
\EX{(M^{(k)}_{\beta,1})^{-s}} \leq C.
\end{align*}
In particular, by Markov's inequality, this implies that for all $\delta>0$ and all $N\geq N_0$, 
\begin{align*}
    \PR{M^{(k)}_{\beta,1} \leq \delta} \leq C \cdot \delta^s.
\end{align*}
\end{lemma}
\begin{proof}
Let $k\in\llbracket 0,K\rrbracket$. By Lemma~\ref{chap3.lem:near zero}, there exist $N_1=N_1(A,\beta)\in\N$ and $C_1>0$ such that for all $N\geq N_1$,
\begin{align*}
N\left(A\left(\frac{k+1}{N}\right)-A\left(\frac{k}{N}\right)\right) 
&\leq \hat{A}'(0) + C_1 \frac{(k+1)^\alpha}{N^\alpha} + C_1\frac{1}{N^\alpha} \nonumber \\
&\leq \hat{A}'(0) + C_1 \frac{(K+1)^\alpha+1}{N^\alpha}.
\end{align*}
Since
\begin{align*}
    \limsup_{N\rightarrow\infty}C_1 \frac{(K+1)^\alpha+1}{N^\alpha} = 0,
\end{align*}
in particular, there exists $N_2=N_2(A,\beta)\in\N$ such that for all $N\geq N_2$,
\begin{align*}
N\left(A\left(\frac{k+1}{N}\right)-A\left(\frac{k}{N}\right)\right) \leq \hat{A}'(0) + 1.
\end{align*}
Let $N_0\coloneqq \max\{N_1,N_2\}$ and
\begin{align*}
C= C(A,\beta, s)\coloneqq \exp(\frac{\beta^2 s^2}{2} (\hat{A}'(0)+1)) 
\cdot 2^s \exp(\frac{\beta^2 s}{2} (\hat{A}'(0)+1)).
\end{align*}
Note that the constant above is greater or equal to $1$ as $\frac{\beta^2 s^2}{2} (\hat{A}'(0)+1)$, $s$ and $\frac{\beta^2 s}{2} (\hat{A}'(0)+1)$ are non negative.
For all $N\geq N_0$, we conclude that 
\begin{align*}
&\EX{(M^{(k)}_{\beta,1})^{-s}} \\
&= \EX{e^{-\beta s X^{(k)}_{1}}}\cdot \EX{Z^{(k)}_{\beta,1}}^s \\
&= \exp(\frac{\beta^2s^2}{2}N\left(A\left(\frac{k+1}{N}\right)-A\left(\frac{k}{N}\right)\right))\cdot 2^s \exp(\frac{\beta^2s}{2}N\left(A\left(\frac{k+1}{N}\right)-A\left(\frac{k}{N}\right)\right)) \\
&\leq \exp(\frac{\beta^2 s^2}{2} (\hat{A}'(0)+1)) 
\cdot 2^s \exp(\frac{\beta^2 s}{2} (\hat{A}'(0)+1)) = C,
\end{align*}
and this completes the proof.
\end{proof}


We now turn to the proof of Proposition~\ref{chap3.prop:bootstrap}.

\begin{proof}[Proof of Proposition~\ref{chap3.prop:bootstrap}]
Fix $\gamma \in (11/10,2)$ and $s>0$. Let $\eta_0=\eta(A,\beta)<1$ be the same as in Proposition~\ref{chap3.prop:init}, and for all $k\in\llbracket 1,K\rrbracket$, define $\eta_k = (\eta_0)^{\gamma^k}$. 
Let $\varepsilon_0=\frac{1}{2}$.  With $C=C(A,\beta,10s)\geq 1$ being the same constant appeared in Lemma~\ref{chap3.lem:one step}, for all $k\in\llbracket 1,K\rrbracket$, define $\varepsilon_k$ such that 
\begin{align*}
    (\varepsilon_k)^s = \frac{1}{2^s}\frac{1}{C^k}\prod_{n=0}^{k-1}((\eta_n)^{\gamma/2}-\eta_n)^{1/10}.
\end{align*}

Firstly, we claim that for all $k\in\llbracket 0,K\rrbracket$, $\ell\in\llbracket 0,K-k\rrbracket$ and $N\geq N_0$, where $N_0=N_0(A,\beta)$ appeared in the statement of Proposition~\ref{chap3.prop:init},
\begin{align}
    \PR{Z_{\beta,N-\ell}^{(\ell)} \leq \varepsilon_k \EX{Z_{\beta,N-\ell}^{(\ell)}}} \leq \eta_k. \label{chap3.eq:good inequality}
\end{align}
Note that \eqref{chap3.eq:good inequality} implies the (i) in the statement of Proposition~\ref{chap3.prop:bootstrap} by taking $\ell=0$.

We prove \eqref{chap3.eq:good inequality} by induction on $k$. The case where $k=0$ follows from Proposition~\ref{chap3.prop:init}. 
Suppose that \eqref{chap3.eq:good inequality} is true for $k\in\llbracket 0,K-1\rrbracket$. 
As
\begin{align*}
    \varepsilon_{k+1} = \varepsilon_{k}\left(\frac{1}{C}(\eta_{k}^{\gamma/2}-\eta_{k})\right)^{\frac{1}{10s}},
\end{align*}
applying Lemma~\ref{chap3.lem:bootstrap} with $c=\varepsilon_{k}$ and $\delta=\left(\frac{1}{C}(\eta_{k}^{\gamma/2}-\eta_{k})\right)^{\frac{1}{10s}}$ yields that for all $\ell\in\llbracket 0,K-k-1\rrbracket$ and $N\geq N_0$,
\begin{align}
    &\PR{Z_{\beta,N-\ell}^{(\ell)}\leq \varepsilon_{k+1} \EX{Z_{\beta,N-\ell}^{(\ell)}}} \nonumber \\
    =
    &\PR{Z_{\beta,N-\ell}^{(\ell)}\leq \varepsilon_{k}\cdot \Big(\frac{1}{C}(\eta_{k}^{\gamma/2}-\eta_{k})\Big)^{1/s} \EX{Z_{\beta,N-\ell}^{(\ell)}}} \nonumber \\
    \leq &\left(\PR{Z_{\beta,N-\ell-1}^{(\ell+1)}\leq \varepsilon_{k} \EX{Z_{\beta,N-\ell-1}^{(\ell+1)}}} + \PR{M_{\beta,1}^{(\ell)} \leq \left(\frac{1}{C}(\eta_{k}^{\gamma/2}-\eta_{k})\right)^{\frac{1}{10s}}} \right)^2
    \nonumber \\
    \leq& (\eta_{k}+\eta_{k}^{\gamma/2}-\eta_{k})^2 
    = (\eta_{k})^{\gamma} 
    = \eta_{k+1}, \label{chap3.eq:eta_r}
\end{align}
where the inequality in \eqref{chap3.eq:eta_r} follows from the induction hypothesis and Lemma~\ref{chap3.lem:one step}.
Thus, by induction, \eqref{chap3.eq:good inequality} holds for all $k\in\llbracket 0,K\rrbracket$.

It remains to show that for all $s>0$, $\sum_{k=0}^K (\varepsilon_{k+1})^{-s}\eta_{k}$ is uniformly bounded in $K$. For all $k\in\llbracket 0,K\rrbracket$,
\begin{align}
    (\varepsilon_{k+1})^{-s}
    &= 2^{s}C^{k+1}\prod_{n=0}^k (\eta_n^{\gamma/2}-\eta_n)^{-1/10} \nonumber \\
    &= 2^{s}C^{k+1}\prod_{n=0}^k \eta_n^{-\gamma/20}(1-\eta_n^{1-\gamma/2})^{-(k+1)/10} \nonumber \\
    &\leq 2^{s}C^{k+1} (1-\eta_0^{1-\gamma/2})^{-1/10}\prod_{n=0}^k \eta_n^{-\gamma/20} \label{chap3.eq:eps_n.0} \\
    &= 2^{s}C^{k+1} (1-\eta_0^{1-\gamma/2})^{-(k+1)/10}\eta_0^{-\frac{\gamma}{20}\frac{\gamma^{k+1}-1}{\gamma-1}} \nonumber \\
    &= 
    C_1 C^{k+1} (1-\eta_0^{1-\gamma/2})^{-(k+1)/10}
    \eta_0^{-\frac{\gamma^2}{20(\gamma-1)}\gamma^k}
    \label{chap3.eq:eps_n}
\end{align}
where $C_1=C_1(s,\eta_0,\gamma)=2^{s}\eta_0^{\frac{\gamma}{20(\gamma-1)}}$ 
and
\eqref{chap3.eq:eps_n.0} follows from the definition of $\eta_n$ and the fact that $(\eta_k)$ is a decreasing sequence which implies in particular that $\eta_k\leq \eta_0$ for all $k\in\llbracket 0,K\rrbracket$.
Applying \eqref{chap3.eq:eps_n}, 
\begin{align}
    &\sum_{k=0}^{K-1} (\varepsilon_{k+1})^{-s}\eta_{k} \nonumber \\
    \leq& C_1\sum_{k=0}^{K-1} 
    C^{k+1} (1-\eta_0^{1-\gamma/2})^{-(k+1)/10}
    \eta_0^{-\frac{\gamma^2}{20(\gamma-1)}\gamma^k} \eta_k \nonumber \\
    =& C_1\sum_{k=0}^{K-1} 
    C^{k+1} (1-\eta_0^{1-\gamma/2})^{-(k+1)/10}
    \eta_0^{(1-\frac{\gamma^2}{20(\gamma-1)})\gamma^k} \label{eq:neg moment.20} \\
    \leq& C_1\sum_{k=0}^{\infty} 
    C^{k+1} (1-\eta_0^{1-\gamma/2})^{-(k+1)/10}
    \eta_0^{(1-\frac{\gamma^2}{20(\gamma-1)})\gamma^k} \nonumber \\
    \leq&
    C_1 \sum_{k=0}^\infty \exp(\log(\frac{C}{(1-\eta_0^{1-\gamma/2})^{1/10}})(k+1)+(1-\frac{\gamma^2}{20(\gamma-1)})(\log\eta_0)\gamma^k)
    \label{chap3.eq:neg moment.2}   
\end{align}
where \eqref{eq:neg moment.20} follows from the definition of $\eta_k$.
Since $\gamma\in (11/10,2)$ and $\eta_0 < 1$, $1-\frac{\gamma^2}{20(\gamma-1)}>0$ and $\log \eta_0 < 1$. Therefore, 
\eqref{chap3.eq:neg moment.2} depends only on $s$, $\gamma\in (11/10,2)$ and $\eta_0=\eta_0(A,\beta)$, and we conclude that $\sum_{k=0}^K (\varepsilon_{k+1})^{-s}\eta_{k}$ is bounded from above by a constant $C_2=C_2(A,\beta,s)>0$ as desired.
\end{proof}

\section{Proof of Theorem~\ref{chap3.thm:main}} \label{chap3.sec:proof of main}

Suppose that $\beta<\beta_c$. Let $N\in\N$. Fix $K=\floor{2 \log_\gamma N}$ and $s>0$. 
Also, we define $W_{\beta,N}\coloneqq Z_{\beta,N}/\EX{Z_{\beta,N}}$.
It suffice to prove that there exists a constant $C=C(A,\beta,s)>0$ such that 
\begin{align*}
\EX{(W_{\beta,N})^{-s}} \leq C.
\end{align*}
We have
\begin{align}
    &\EX{(W_{\beta,N})^{-s}} \nonumber \\
    &\leq \varepsilon_0^{-s}\PR{W_{\beta,N}>\varepsilon_0} + \sum_{k=0}^{K-1} (\varepsilon_{k+1})^{-s} \cdot \PR{\varepsilon_{k+1} \leq W_{\beta,N}\leq \varepsilon_k}
    + \EX{(W_{\beta,N})^{-s}\Ind_{W_{\beta,N}\leq \varepsilon_K}} \nonumber \\
    &\leq 2^s + \sum_{n=0}^{K-1} (\varepsilon_{k+1})^{-s} \cdot \PR{W_{\beta,N}\leq \varepsilon_k}
    + \EX{(W_{\beta,N})^{-s}\Ind_{W_{\beta,N}\leq \varepsilon_K}}.
    \label{chap3.eq:neg moment.1}
\end{align}
By Proposition~\ref{chap3.prop:bootstrap}, there exist $N_1=N(A,\beta,s)\in\N$ and $C_1=C_1(A,\beta,s)>0$ such that for all $N\geq N_1$ the second term of \eqref{chap3.eq:neg moment.1} is bounded by $C_1$. Thus, it remains to bound the third term of \eqref{chap3.eq:neg moment.1}. 

For all $\abs{u}=N$,
\begin{align}
    Z_{\beta,N} \geq e^{\beta X_u}. \label{chap3.eq:one particle}
\end{align}
Fix a vertex $u$ with $\abs{u}=N$. By \eqref{chap3.eq:one particle} and Lemma~\ref{lem:exponential tilting}, the second term of \eqref{chap3.eq:neg moment.1} is bounded from above by 
\begin{align}
    \EX{(W_{\beta,N})^{-s}\Ind_{W_{\beta,N}\leq \varepsilon_K}}
    &\leq \EX{Z_{\beta,N}}^{s}\cdot \EX{(e^{\beta X_u})^{-s}\Ind_{e^{\beta X_u} \leq \varepsilon_K\, \EX{Z_{\beta,N}}}} \nonumber \\
    &= \EX{Z_{\beta,N}}^{s}\cdot \EX{e^{-\beta s X_u}\Ind_{X_u \leq \frac{1}{\beta}(\log \varepsilon_K + \log \EX{Z_{\beta,N}})}} \nonumber \\
    &= \EX{Z_{\beta,N}}^{s} \EX{e^{-\beta s X_u}\Ind_{X_u \leq \frac{1}{\beta}(\log \varepsilon_K + N\log 2 + N\frac{\beta^2}{2})}} \nonumber \\
    &= \exp(\frac{\beta^2sN}{2})
    \EX{e^{-\beta s X_u}\Ind_{X_u \leq \frac{1}{\beta}(\log \varepsilon_K + N\log 2 + N\frac{\beta^2}{2})}}
    \label{eq:two s moments}
\end{align}
Note that $X_u$ appearing in the expectation of \eqref{eq:two s moments} is distributed as a centered Gaussian variable with variance $N$, so the second expectation of \eqref{eq:two s moments} can be rewritten as
\begin{align}
    &\EX{e^{-\beta s X_u}\Ind_{X_u \leq \frac{1}{\beta}(\log \varepsilon_K + N\log 2 + N\frac{\beta^2}{2})}} \nonumber \\
    &=
    \frac{1}{\sqrt{2\pi N}}
    \int_{-\infty}^{\frac{1}{\beta}(\log \varepsilon_K + N\log 2 + N\frac{\beta^2}{2})}
    e^{-\beta s x}e^{-\frac{x^2}{2N}} \dd{x} \nonumber \\
    &=e^{\frac{\beta^2s^2}{2}N}
    \frac{1}{\sqrt{2\pi N}}
    \int_{-\infty}^{\frac{1}{\beta}(\log \varepsilon_K + N\log 2 + N\frac{\beta^2}{2})}
    e^{-\frac{(x+\beta s N)^2}{2N}} \dd{x} \nonumber \\
    &=e^{\frac{\beta^2s^2}{2}N}
    \frac{1}{\sqrt{2\pi N}}
    \int_{-\frac{1}{\beta\sqrt{N}}(\log \varepsilon_K + N\log 2 + N\frac{\beta^2}{2}+N \beta s)}^\infty
    e^{-\frac{y^2}{2}} \dd{y} \label{eq:Gaussian integral} \\
    &= 
    e^{\frac{\beta^2s^2}{2}N}
    \frac{1}{\sqrt{2\pi N}}
    \int_{r(N)}^\infty
    e^{-\frac{y^2}{2}} \dd{y},  \label{eq:def of r(N)}
\end{align}
where \eqref{eq:Gaussian integral} follows from the change of variables $y=(x+\beta s)/\sqrt{N}$ and $r(N)$ in \eqref{eq:def of r(N)} is defined as $r(N)=-\frac{1}{\beta\sqrt{N}}(\log \varepsilon_K + N\log 2 + N\frac{\beta^2}{2}+N \beta s)$.

Recall the definition of $(\varepsilon_k)_{k=0}^K$ in the proof of Proposition~\ref{chap3.prop:bootstrap}, we have
\begin{align}
    \varepsilon_K
    = \frac{1}{2}\frac{1}{C^K}\prod_{n=0}^{K-1}((\eta_n)^{\gamma/2}-\eta_n)^{1/(10s)}
    \leq 
    \prod_{n=0}^{K-1}\eta_n^{\gamma/(20s)}
    =
    \prod_{n=0}^{K-1}\eta_0^{\gamma^{n+1}/(20s)}
    , \label{eq:eps bound}
\end{align}
where the inequality above follows from the fact that $1/2\geq 1$, $C\geq 1$ and $\eta_n\geq 0$ for all $n\in\llbracket 0,K-1\rrbracket$. Recall that $K=\floor{2\log_\gamma N}$, so
\begin{align}
    \eqref{eq:eps bound}
    =
    \eta_0^{\frac{\gamma}{20s}\frac{\gamma^K-1}{\gamma-1}}
    \leq 
    \eta_0^{\frac{\gamma}{20(\gamma-1)s}N^2}
    . \label{eq:eps bound.1}
\end{align}
Then, \eqref{eq:eps bound} and \eqref{eq:eps bound.1} imply that
\begin{align}
    r(N)
    &\geq 
    -\frac{1}{\beta\sqrt{N}}
    \left( 
    \frac{\gamma}{20(\gamma-1)s}(\log \eta_0)N^2
    + N\log 2 + N\frac{\beta^2}{2}+N \beta s
    \right) \nonumber \\
    &= c_1 N^{3/2} - (\frac{\log 2}{\beta} + \frac{\beta}{2} + s) N^{1/2},
\end{align}
and since $\eta_0<1$, we have
\begin{align}
    c_1 = -\frac{\gamma\log(\eta_0)}{20\beta(\gamma-1)s} >0.
\end{align}
Therefore, there exists $N_2=N_2(A,\beta,s)$ such that for all $N\geq N_2$,
$r(N)\geq \frac{c_1}{2} N^{3/2}$, and this implies that for $N\geq N_2$, we have
\begin{align}
    \int_{r(N)}^\infty
    e^{-\frac{y^2}{2}} \dd{y}
    \leq 
    \int_{\frac{c_1}{2}N^{3/2}}^\infty e^{-\frac{x^2}{2}} \dd{x}
    \leq 
    \frac{1}{(c_1/2) N^{3/2}} e^{-\frac{c_1^2N^3}{8}}, \label{eq:Gaussian tail}
\end{align}
where the second inequality follows from a standard estimate for Gaussian tails (see, for example, Theorem~1.2.6~of~\cite{Durret}). Combining \eqref{eq:two s moments}, \eqref{eq:def of r(N)} and \eqref{eq:Gaussian tail}, we see that 
\begin{align}
     &\limsup_{N\rightarrow\infty}
     \EX{(W_{\beta,N})^{-s}\Ind_{W_{\beta,N}}} \nonumber \\
     &\leq 
     \limsup_{N\rightarrow\infty}
     \left(
     e^{\frac{\beta^2sN}{2}} \cdot e^{\frac{\beta^2s^2}{2}N}
    \frac{1}{\sqrt{2\pi N}}
    \cdot 
    \frac{1}{(c_1/2) N^{3/2}} e^{-\frac{c_1^2N^3}{8}}
    \right)
    = 0. \label{eq:second term limsup}
\end{align}
From \eqref{chap3.eq:neg moment.1} and \eqref{eq:second term limsup}, we conclude that there exists $N_0 = N_0(A,\beta,s)\geq \max\{N_1,N_2\}$ such that for $N\geq N_0$,
\begin{align*}
\EX{(W_{\beta,N})^{-s}}\leq 2^s + C_1 + 1.
\end{align*}
By taking $C = 2^s + C_1+1$, the proof is completed.

\subsubsection*{Acknowledgments}

I would like to thank Pascal Maillard for his guidance throughout the whole project and Michel Pain for his comments on the early draft of the paper. I also appreciate the constructive suggestions from an anonymous referee, which have substantially improved the quality of the paper.

\end{document}